\documentclass[12pt,a4paper]{article}

\usepackage{amssymb, amsmath, amsthm, amssymb}

\usepackage[cm]{fullpage}
\usepackage[english]{babel}
\usepackage{booktabs}
\usepackage{xcolor,enumerate}

\usepackage{hyperref}
\usepackage{autonum}

\numberwithin{equation}{section}

\newtheorem{thm}{Theorem}
\newtheorem{lem}{Lemma}

\newtheorem{rem}{Remark}

\title{A note on the L1 discretization error for the Caputo derivative in H\"older spaces\footnote{This is an accepted version of a manuscript published in Applied Mathematics Letters \textbf{161} (2025), 109364 with DOI: \url{https://doi.org/10.1016/j.aml.2024.109364}}}

\author{F\'elix del Teso\thanks{Departamento de Matem\'aticas, Universidad Aut\'onoma de Madrid, Spain, felix.delteso@uam.es}, \; \L ukasz P\l ociniczak\thanks{Faculty of Pure and Applied Mathematics, Wroc{\l}aw University of Science and Technology, Poland, lukasz.plociniczak@pwr.edu.pl}}

\date{}

\begin{document}
\maketitle

\begin{abstract}
	We establish  uniform error bounds of the L1 discretization of the Caputo derivative of H\"older continuous functions. The result can be understood as: \textit{error = degree of smoothness - order of the derivative.} We present an elementary proof and illustrate its optimality with numerical examples.\\
	
\end{abstract}

\section{Introduction}
The Caputo derivative is defined for sufficiently smooth functions by
\begin{equation}\label{eqn:Caputo2}
D^\alpha y(t)=\frac{1}{\Gamma(1-\alpha)}\int_0^t \frac{y'(s)}{(t-s)^\alpha} dt=\frac{1}{\Gamma(1-\alpha)}\frac{y(t)-y(0)}{t^{\alpha}}+\frac{\alpha}{\Gamma(1-\alpha)}\int_0^t \frac{y(t)-y(s)}{(t-s)^{1+\alpha}} ds, \quad 0<\alpha<1.
\end{equation}
Equivalent between formulations can be recovered integrating by parts.
Probably most of the methods used to discretize the Caputo derivative are derived or generalized from the Convolution Quadrature (utilizing the convolutional form of the Caputo derivative, \cite{lubich2004convolution}), Gr\"unwald-Letnikov (generalization of finite difference, \cite{gorenflo2007convergence}), or the L1 scheme (piecewise linear approximation, \cite{stynes2022survey}). The latter can be written in the following way
\begin{equation}\label{eqn:L1Scheme}
\delta^\alpha_\tau y(t_n) := \frac{\tau^{-\alpha}}{\Gamma(2-\alpha)} \left(y(t_n) - b_{n-1} y(0) - \sum_{i=1}^{n-1}\left(b_{n-i-j}-b_{n-i}\right)y(t_i)\right), \quad b_i = (i+1)^{1-\alpha} - i^{1-\alpha},
\end{equation}
through which the Caputo derivative is approximated on the discrete grid $t_n = n \tau \leq T$ with the step $\tau > 0$ and the maximal time $T>0$. The L1 method has many pleasant properties that make it versatile and useful in devising numerical methods for different purposes.

Although widely used in many contexts, the L1 method still misses some analytical results when it comes to estimating the truncation error. In this note, we are concerned with H\"older continuous functions and the corresponding performance of the L1 method. Historically the first proof of the accuracy of the discretization assumed that 
$y\in C^2[0,T]$ resulting in the order of $2-\alpha$ (see for example \cite{li2019theory}, Theorem 4.1). In this paper, we will find uniform bounds on the discretization error of the L1 method for functions belonging to a H\"older space $C^{k,\beta}[0,T]$ with $k + \beta > \alpha$. The main result can be diagrammatically illustrated as:
\emph{order of the truncation error} $= k + \beta - \alpha.$
Note that, for $C^{1,1}[0,T]$ functions we recover the well-known optimal order $2-\alpha$. 

\section{The main result}
The space of H\"older functions is defined in a standard way. For $k \in \mathbb{N}$ and $0<\beta\leq 1$ we define the \textit{H\"older space} by
\begin{equation}
C^{k,\beta}[0,T] := \left\{y \in C^k[0,T]: \, [y^{(k)}]_{C^{0,\beta}[0,T]} < \infty\right\} \quad \text{where} \quad [y]_{C^{0,\beta}[0,T]} := \sup_{ t\not= \tilde{t}, \ t,\tilde{t}\in [0,T]} \frac{|y(t)-y(\tilde{t})| }{|t-\tilde{t}|^{\beta}}.
\end{equation}
\begin{rem}
The Caputo derivative is naturally well-defined on $C^{0,\beta}[0,T]$ with $\beta > \alpha$. Indeed, we have
\begin{equation}
	\begin{split}
		|D^\alpha y(t)|&\leq\frac{1}{\Gamma(1-\alpha)}\frac{|y(t)-y(0)|}{t^{\alpha}}+\frac{\alpha}{\Gamma(1-\alpha)}\int_0^t \frac{|y(t)-y(s)|}{(t-s)^{1+\alpha}} ds \\
		&\leq [y]_{C^{0,\beta}[0,t]} \left(\frac{t^{\beta-\alpha}}{\Gamma(1-\alpha)}+\frac{\alpha}{\Gamma(1-\alpha)}\int_0^t (t-s)^{\beta-\alpha-1} ds\right)<+\infty.
	\end{split}
\end{equation} 
\end{rem}
The main result of the paper can now be stated as follows.

\begin{thm}\label{thm:Truncation}
Let $y \in C^{k,\beta}[0,T]$ with $k=0,1$ and $0\leq \beta \leq 1$ with $k+\beta > \alpha$. Then, the truncation error of the L1 discretization \eqref{eqn:L1Scheme} for $t_n\in (0, T]$ satisfies
\begin{equation}
	|D^\alpha y(t_n) - \delta^\alpha_\tau y(t_n)| \leq C_{\alpha,\beta, k,n} [y^{(k)}]_{C^{0,\beta}[0,t_n]} \tau^{k+\beta-\alpha},
\end{equation}
where the error constant has an explicit form
\begin{equation}
	C_{\alpha,\beta, k, n} = \begin{cases}
		\dfrac{1-n^{-\alpha}}{2^\beta\Gamma(1-\alpha)} + \dfrac{\alpha}{\Gamma(1-\alpha)(\beta-\alpha)(1+\beta-\alpha)} + \dfrac{\alpha\Gamma(1+\beta)}{\Gamma(2+\beta-\alpha)}, & k = 0, \vspace{4pt}\\
		\dfrac{1-n^{-\alpha}}{4(\beta+1)\Gamma(1-\alpha)}+ \dfrac{\alpha}{\Gamma(3-\alpha)} , & k = 1. \\
	\end{cases}
\end{equation} 
\end{thm}
\begin{rem}
The obtained order of convergence $k+\beta-\alpha$ is optimal under the assumptions of Theorem \ref{thm:Truncation}. This is illustrated in Section 3 with numerical example. 
\end{rem}
\begin{rem}
The above estimate is $\alpha$-robust and $n$-robust. 
To be more precise, on the one hand, if $k+\beta \geq 1$, that is, either $k=0$ and $\beta=1$ or $k=1$ and $\beta \in [0,1]$, then
$\sup_{n \in \mathbb{N}, \, \alpha \in (0,1)}C_{\alpha,\beta,k,n} < +\infty$, $\lim_{\alpha \to 0^+} C_{\alpha,\beta,k,n} = 0$ $\lim_{\alpha \to 1^-} C_{\alpha,1,0,n} = 2$ and  $\lim_{\alpha \to 1^-} C_{\alpha,\beta,1,n} = 1$. On the other hand, if $k+\beta < 1$, that is, $k=0$ and $\beta \in (0,1)$, then, for all $\alpha_0 \in (0,\beta)$, we have that
$\sup_{n \in \mathbb{N}, \, \alpha \in (0,\alpha_0)} C_{\alpha,\beta,0,n} < +\infty$ and $\lim_{\alpha \to 0^+} C_{\alpha,\beta,0,n} = 0$.
Note that, in this latter case, we exclude the limit $\alpha \to \beta^-$ from the robustness discussion. This is due to the fact that, in general, for a $\beta$-Hölder function, the Caputo derivative of order $\beta$ may fail to exist (for example, $D^{\beta} y(1)$ is not well defined for $y(t) = (1-t)^\beta$).
%
%

\end{rem}

The central tool that we will use in the proof of the above is the piecewise linear interpolation operator. 
It is defined as assigning a linear interpolant to each subinterval $[t_j, t_{j+1})$. That is, for a continuous function $y$, we have
\begin{equation}\label{eqn:Interpolation}
I_\tau y(t) = \frac{t-t_j}{\tau} y(t_{j+1}) + \frac{t_{j+1}-t}{\tau} y(t_{j}), \quad t\in [t_j, t_{j+1}), \quad j\in\mathbb{N}.
\end{equation}
\begin{rem}\label{rem:defl1}
By definition, the L1 discretization is a Caputo derivative of a piecewise linear interpolant. This can be verified by substituting \eqref{eqn:Interpolation} instead of $y(t)$ into \eqref{eqn:Caputo2}, to obtain \eqref{eqn:L1Scheme}. That is, $D^\alpha I_\tau y(t_n) = \delta_\tau^\alpha y(t_n)$.
\end{rem}
Let us establish some notation: given a function $y \in C([0,t_n])$ we define its modulus of continuity $\Lambda_{y}:[0,T]\to \mathbb{R}$ by
\begin{equation}
\Lambda_y(\delta):= \sup_{|t-\tilde{t}|\leq \delta, \ t,\tilde{t}\in [0,t_n]} |y(t)-y(\tilde{t})|.
\end{equation}
\begin{rem}\label{rem:modholder}
$\Lambda_y$ is nondecreasing, $\Lambda(\delta)\to 0^+$ as $\delta \to 0^+$ and, if $y\in C^{0,\beta}([0,T])$ for some $\beta\in (0,1]$, then
\[
\Lambda_y(\delta)\leq [y]_{C^{0,\beta}[0,t_n]} \delta^{\beta}.
\]
\end{rem}
The following key lemma gives the error of interpolation for H\"older functions in the pointwise and in the uniform sense. We do not claim any novelty of the following result; however, we wanted to include it here for completeness and for the precise statement needed for proof of the main result. 
\begin{lem}\label{lem:Interpolation}
Let $y\in C^{k, \beta}([0,t_n])$ with $k=0,1$ and $\beta\in[0,1]$. Then, for $t\in [t_j, t_{j+1})$ with $j=0,\ldots,n-1$ we have
\begin{equation}
	|I_\tau y(t) - y(t)|\leq 
	\begin{cases}
		\left((t-t_j) \Lambda_y(t_{j+1}-t) + (t_{j+1}-t) \Lambda_y(t-t_j)\right)\tau^{-1} &\textup{if} \quad k=0 \quad \text{and} \quad \beta=0,\\
		\left((t-t_j) (t_{j+1}-t)^\beta + (t_{j+1}-t)(t-t_j)^\beta\right)\tau^{-1} [y]_{C^{0,\beta}[0,t_n]} &\textup{if} \quad k=0 \quad \text{and} \quad \beta\in (0,1],\\
		(t-t_j)(t_{j+1}-t)\Lambda_{y'}(\tau)\tau^{-1} &\textup{if} \quad k=1 \quad \text{and} \quad \beta=0,\\
		(t-t_j)(t_{j+1}-t)(\beta+1)^{-1}\tau^{\beta-1}[y']_{C^{0,\beta}[0,t_n]} &\textup{if} \quad k=1 \quad \text{and} \quad \beta\in (0,1].
	\end{cases}
\end{equation}
\end{lem}
\begin{proof}
First, suppose that $k=0$ and fix $t\in[t_j, t_{j+1})$. By the observation that the coefficients of $y(t_j)$ and $y(t_{j+1})$ in $I_\tau y(t)$ sum up to $1$ we can write
\begin{equation}\label{eq:interpmod}
	\begin{split}
		|I_\tau y(t) - y (t)| \leq \frac{t-t_j}{\tau} |y(t_{j+1})-y(t)| + \frac{t_{j+1}-t}{\tau} |y(t_{j}) - y(t)|
		\leq \frac{t-t_j}{\tau} \Lambda_y(t_{j+1}-t) + \frac{t_{j+1}-t} {\tau}\Lambda_y(t-t_j).
	\end{split}
\end{equation}
From the above, the estimate for $\beta=0$ follows directly and when $\beta \in (0,1]$, we just use Remark \ref{rem:modholder}.
Now, if $k=1$, we use Taylor's theorem with explicit reminder to get
\begin{equation}
	\begin{split}
		I_\tau y(t) - y (t) &= \frac{t-t_j}{\tau} (y(t_{j+1})-y(t)) + \frac{t_{j+1}-t}{\tau} (y(t_{j}) - y(t))\\
		&= \frac{(t-t_j)(t_{j+1}-t)}{\tau}\int_0^1 \left(y'(\rho t_{j+1}+(1-\rho)t)-y'(\rho t_{j}+(1-\rho)t)\right)d \rho.
	\end{split}
\end{equation}
Taking absolute values in the above identity leads to the following estimate
\begin{equation}\label{eq:interpmod2}
	|I_\tau y(t) - y (t)| \leq  \frac{(t-t_j)(t_{j+1}-t)}{\tau} \int_0^1\Lambda_{y'}(\rho \tau) d\rho.
\end{equation}
Now, if $\beta=0$, we use the fact that $\Lambda_{y'}$ is nondecreasing to get $|I_\tau y(t) - y (t)| \leq (t-t_j)(t_{j+1}-t)\tau^{-1} \Lambda_{y'}(\tau)$. Finally, if $\beta\in(0,1]$, we use \eqref{eq:interpmod2} and Remark \ref{rem:modholder} to obtain the desired result.
\end{proof}
%

Now, we can proceed to the proof of the main result. 

\begin{proof}[Proof of Theorem \ref{thm:Truncation}]
We start with the case $k=0$ (thus,  $\beta>\alpha$). By Remark \ref{rem:defl1} and Lemma \ref{lem:Interpolation} we get

\begin{equation}\label{eqn:TruncationEstimate}
	\begin{split}
		|D^\alpha y(t_n) &- \delta^\alpha_\tau y(t_n)| 
		\leq  \frac{\alpha}{\Gamma(1-\alpha)} \int_0^{t_n} \frac{|I_\tau y(s) - y(s)|}{(t_n-s)^{1+\alpha}} ds = \frac{\alpha}{\Gamma(1-\alpha)} \sum_{j=0}^{n-1} \int_{t_j}^{t_{j+1}} \frac{|I_\tau y(s) - y(s)|}{(t_n-s)^{1+\alpha}} ds \\
		&\leq [y]_{C^{0,\beta}[0,t_n]}\tau^{-1}\frac{\alpha}{\Gamma(1-\alpha)} \sum_{j=0}^{n-1}\underbrace{\int_{t_j}^{t_{j+1}} (t_n-s)^{-1-\alpha} \left((s-t_{j})(t_{j+1}-s)^\beta + (t_{j+1}-s) (s-t_j)^\beta\right)ds}_{K_j}.
	\end{split}
\end{equation}
Since the only singularity occurs in the $j=n-1$ term with $n \geq 1$, the last integral has to be estimated separately by the change of the variable $w = (s-t_{n-1})/\tau$, hence
\begin{equation}
	\begin{split}
		K_{n-1} \leq \tau^{1+\beta-\alpha} \int_{0}^{1} \left(w (1-w)^{\beta-1-\alpha} + w^\beta(1-w)^{-\alpha}\right) dw = \tau^{1+\beta-\alpha}\left(\frac{1}{(\beta-\alpha)(1+\beta-\alpha)} + \frac{\Gamma(1+\beta)\Gamma(1-\alpha)}{\Gamma(2+\beta-\alpha)} \right).
	\end{split}
\end{equation}
where in the last step, we have used the definition of the beta function and its relation with the gamma function. The remaining part can be estimated by maximizing $(s-t_{j})(t_{j+1}-s)^\beta + (t_{j+1}-s) (s-t_j)^\beta$ for $s\in[t_j,t_{j+1})$ and summing the series to get

\begin{equation}
	\begin{split}
		\sum_{j=0}^{n-2} &K_j\leq \frac{\tau^{1+\beta}}{2^\beta} \sum_{j=0}^{n-2} \int_{t_j}^{t_{j+1}} (t_n-s)^{-1-\alpha} ds = \frac{\tau^{1+\beta}}{2^\beta} \int_0^{t_{n-1}} (t_n-s)^{-1-\alpha} ds = \tau^{1+\beta-\alpha} \frac{1-n^{-\alpha}}{2^\beta\alpha}.
	\end{split}
\end{equation}
Combining the both estimates we obtain the desired result. For $k=1$ we again use Lemma \ref{lem:Interpolation}. Similarly as before, we split the last integral from the sum and substitute $w = (s-t_{n-1})/\tau$ in it, while maximizing $(s-t_{j})(t_{j+1}-s)$ in the remaining ones
\begin{equation}
	\begin{split}
		|D^\alpha y(t_n) &- \delta^\alpha_\tau y(t_n)| \leq [y']_{C^{0,\beta}[0,t_n]} \frac{\alpha}{\Gamma(1-\alpha)} \tau^{\beta-1} \sum_{j=0}^{n-1} \int_{t_j}^{t_{j+1}} (t_n-s)^{-1-\alpha} (s-t_j) (t_{j+1}-s) ds \\
		&\leq [y']_{C^{0,\beta}[0,t_n]}\frac{\alpha}{\Gamma(1-\alpha)} \tau^{\beta-1} \left(\tau^{2-\alpha} \int_{0}^{1} (1-w)^{-\alpha} w \, dw + \frac{\tau^2}{4(\beta+1)} \sum_{j=0}^{n-2} \int_{t_j}^{t_{j+1}} (t_n-s)^{-1-\alpha} ds \right) \\
		&=[y']_{C^{0,\beta}[0,t_n]}\frac{\alpha}{\Gamma(1-\alpha)} \tau^{\beta-1} \left(\frac{\tau^{2-\alpha}}{(1-\alpha)(2-\alpha)} + \frac{\tau^2}{4(\beta+1)} \int_{0}^{t_{n-1}} (t_n-s)^{-1-\alpha} ds \right) \\ 
		&=\left(\frac{\alpha}{\Gamma(3-\alpha)} + \frac{1-n^{-\alpha}}{4(\beta+1)\Gamma(1-\alpha)} \right)[y']_{C^{0,\beta}[0,t_n]} \tau^{1+\beta-\alpha},
	\end{split}
\end{equation}
and the proof is completed. 
\end{proof}


\begin{rem}
The most regular case with $k=1$ and $\beta=1$ gives the optimal error of order $2-\alpha$ that coincides with the previously known case of $y\in C^2[0,T]$ (see \cite{li2019theory}), that is, 
$
	|D^\alpha y(t_n) - \delta^\alpha_\tau y(t_n)|\leq \tilde{C}_{\alpha,n}\tau^{2-\alpha} \max_{t\in[0,T]} |y''(t)|.
$
It is also possible to find the explicit asymptotic form of the best error constant $\tilde{C}_{\alpha,n}\sim -\zeta(\alpha-1)/\Gamma(2-\alpha)$ where $\zeta$ denotes the Riemann zeta function (see \cite{plociniczak2023linear}). In particular, this shows that the constant $C_{\alpha,\beta,k,n}$ from Theorem \ref{thm:Truncation} is not optimal since it is always strictly larger than $\tilde{C}_{\alpha,n}$ for $k=1$ and $\beta=1$.  

Graded meshes can also be employed to maintain the optimal convergence rate of $2-\alpha$ while relaxing the regularity requirement to $y \in C^\alpha((0,T))$, with $y^{(m)}(t) \sim t^{\alpha-m}$ for $m = 0, 1, 2$ (see \cite{stynes2017error,kopteva2019error}). These assumptions align naturally with the typical behavior of solutions to fractional differential equations. Since we are considering a general setting, we do not impose any additional blow-up conditions (see for example \eqref{eqn:TestFunction}). Our calculations can similarly be extended to the case of a graded mesh; however, without further information about the function's behavior, the order of convergence estimates cannot be improved.
\end{rem}

\section{Numerical illustration}
As the test function on which we would like to verify results from the previous section we choose the following
\begin{equation}
\label{eqn:TestFunction}
	y_{k,\beta}(t) := \left(t-\frac{1}{2}\right)^k\left|t-\frac{1}{2}\right|^{\beta}, \quad k = 0,1, \quad \beta\in(0,1].
\end{equation}
The order of convergence of the L1 scheme can be estimated via the extrapolation, that is
\begin{equation}
	\text{order} \approx \log_2 \frac{\max_{t\in[0,T]}|\delta^\alpha_\tau y(t) - \delta^\alpha_{\tau/2}y(t)|}{\max_{t\in[0,T]}|\delta^\alpha_{\tau/2} y(t) - \delta^\alpha_{\tau/4}y(t)|},
\end{equation} 
where we refine the grid in order to infer about the exact error. In Table \ref{tab:Orders} we have gathered the results of our computations rounded to the second decimal place. We immediately observe that the estimated order of convergence agrees with the theoretical value $k+\beta-\alpha$ exceptionally well for all the choices of $\alpha$, $\beta$, and $k$. Moreover, for the case $k + \beta < \alpha$ numerical simulations yield a negative order, meaning lack of convergence. This again is in accord with the well-posedness of the Caputo derivative discussed at the beginning of Section 2. Note also that even in this case, the expression for the order equal to $k+\beta-\alpha$ is preserved regardless of the assumption that $\beta > \alpha$!

\begin{table}[h!]
	\centering
	\begin{tabular}{c|cccccccccc}
		\toprule
		$\alpha \; \backslash \; k+\beta$ & 0.1 & 0.3 & 0.5 & 0.7 & 0.9 & 1.1 & 1.3 & 1.5 & 1.7 & 1.9 \\
		\midrule
		0.1 & 0.0 & 0.2 & 0.4 & 0.6 & 0.8 & 1.0 & 1.2 & 1.4 & 1.59 & 1.83 \\
		0.3 & -0.2 & 0.0 & 0.2 & 0.4 & 0.6 & 0.8 & 1.0 & 1.2 & 1.4 & 1.68 \\
		0.5 & -0.4 & -0.2 & 0.0 & 0.2 & 0.4 & 0.6 & 0.8 & 1.0 & 1.2 & 1.49 \\
		0.6 & -0.6 & -0.4 & -0.2 & 0.0 & 0.2 & 0.4 & 0.6 & 0.8 & 1.0 & 1.29 \\
		0.9 & -0.8 & -0.6 & -0.4 & -0.2 & 0.0 & 0.2 & 0.4 & 0.6 & 0.8 & 1.0 \\
		\bottomrule
\end{tabular}

\caption{Estimated orders of convergence of the L1 scheme for various choices of $\alpha$ and $\beta$ with $T = 1$. The test function is taken to be \eqref{eqn:TestFunction}. The base of calculations is $\tau = 2^{-10}$. }
\label{tab:Orders}

\end{table}

\section*{Acknowledgement}
FdT was supported by the Spanish Government through RYC2020-029589-I, PID2021127105NB-I00 and CEX2019-000904-S funded by the MICIN/AEI. ŁP has been supported by the National Science Centre, Poland (NCN) under the grant Sonata Bis with a number NCN 2020/38/E/ST1/00153. 


\end{document}